\documentclass[sn-mathphys-num]{sn-jnl}

\usepackage{graphicx}%
\usepackage{multirow}%
\usepackage{amsmath,amssymb,amsfonts}%
\usepackage{amsthm}%
\usepackage{mathrsfs}%
\usepackage[title]{appendix}%
\usepackage{xcolor}%
\usepackage{textcomp}%
\usepackage{manyfoot}%
\usepackage{booktabs}%
\usepackage{algorithm}%
\usepackage{algorithmicx}%
\usepackage{algpseudocode}%
\usepackage{listings}
\usepackage{graphicx}
\usepackage{tikz}
\usepackage{float}
\usepackage{caption}
\usepackage{enumerate}

\def\eps{\varepsilon }
\def\CC{\mathbb C}
\def\RR{\mathbb R}
\def\ZZ{\mathbb Z}
\def\TT{\mathbb T}

\newcommand{\set}[1]{\left\lbrace #1\right\rbrace}
\providecommand{\abs}[1]{\left\lvert#1\right\rvert}
\providecommand{\norm}[1]{\left\lVert#1\right\rVert}
\newcommand{\qtq}[1]{\quad\text{#1}\quad}
\newcommand{\vv}[1]{\left\langle#1\right\rangle}
\newcommand{\bj}[1]{\left(#1\right)}

\DeclareMathOperator{\dist}{dist}

\newtheorem{theorem}{Theorem}[section]
\newtheorem{proposition}[theorem]{Proposition}
\newtheorem{lemma}[theorem]{Lemma}
\newtheorem{corollary}[theorem]{Corollary}

\theoremstyle{definition}
\newtheorem*{definition}{Definition}
\newtheorem*{definitions}{Definitions}

\theoremstyle{remark}
\newtheorem{remark}{Remark}
\newtheorem{remarks}{Remarks}
\newtheorem{example}{Example}

\numberwithin{equation}{section}
\begin{document}

\title[Control results for ZK equation]{Observability of the linear
Zakharov--Kuznetsov equation}

\author[1]{\fnm{Roberto} \sur{de A. Capistrano--Filho}}\email{roberto.capistranofilho@ufpe.br}
\author*[2]{\fnm{Vilmos} \sur{Komornik}}\email{komornik@math.unistra.fr}

\author[3]{\fnm{Ademir F.} \sur{Pazoto}}\email{ademir@im.ufrj.br}

\affil[1]{\orgdiv{Departamento de Matemática}, \orgname{Universidade Federal de Pernambuco}, \orgaddress{\street{S/N Cidade Universitária}, \city{Recife}, \postcode{50740-545}, \state{PE}, \country{Brazil}}}

\affil*[2]{\orgdiv{Département de mathématique}, \orgname{Université de Strasbourg}, \orgaddress{\street{7 rue René Descartes}, \city{67084 Strasbourg Cedex}, \country{France}}, orcid number: 0000-0001-8061-7374}

\affil[3]{\orgdiv{Instituto de Matemática}, \orgname{Universidade Federal do Rio de Janeiro}, \orgaddress{\street{Cidade Universitária - Ilha do Fundão}, \city{Rio de Janeiro}, \postcode{,
21945-970}, \state{RJ}, \country{Brazil}}}

\abstract{We study the linear Zakharov--Kuznetsov equation with periodic boundary conditions. Employing some tools from the nonharmonic Fourier series we obtain several internal observability theorems. Then we prove various exact controllability and rapid uniform stabilization results by applying a duality principle and a general feedback construction. The method presented here introduces a new insight into the control of dispersive equations {in two-dimensional cases } and may be adapted to more general equations.
}



\maketitle

\section{Introduction} The Zakharov--Kuznetsov equation
$$
u_t+u_x+u_{xxx}+u_{xyy}+uu_x=0
$$
was originally derived by Zakharov and Kuznetsov \cite{ZK1974} in three dimensions to describe weakly magnetized ion-acoustic waves in a strongly magnetized plasma. The two-dimensional equation still has some physical meaning, modeling the behavior of weakly nonlinear ion-acoustic waves in a plasma comprising cold ions and hot isothermal electrons in the presence of a uniform magnetic field (see \cite{MoPa-99,MoPa-00} for the details and \cite{LaLiSa-13} for rigorous mathematical treatment.)

In this paper, we consider the linearized Zakharov--Kuznetsov (ZK) equation
\begin{equation}\label{11}
u_t+u_x+u_{xxx}+u_{xyy}=0
\end{equation}
on the two-dimensional torus $\TT\times\TT$, where $u(x,y,t)$ denotes a real-valued function of three real variables $x$, $y$ and $t$, and $\TT$ denotes the unit circle. As a higher-dimensional analog of the Korteweg--de Vries equation it has also been extensively studied and the body of literature is huge {(see below)}. In the following, we deliver an overview of the well-posedness and the controllability theory in two dimensions.

\subsection{Historical background}

The Zakharov--Kuznetsov equation and its linearized version are special cases of a broad class of dispersive equations for which the well-posedness theory associated with the pure initial-value problem posed on the whole plane has been intensively investigated. In this respect, we can mention the works \cite{Faminskii-95, Kinoshita-21, LinPas-09, MolPil-15} and the references therein. As for the periodic case, the authors of the pioneering work \cite{LinPanRo-19} proved that the initial value problem for the Zakharov--Kuznetsov equation is locally well-posed in the Sobolev spaces $H^s(\mathbb{T}\times\mathbb{T})$ for $s > 5/3$. Their main tool was a variant of Strichartz's estimate. By using a nonlinear Loomis--Whitney type inequality, the assumption $s > 5/3$ was weakened to $s>1$ in \cite{KinSch-21}. More recently, it was shown by Osawa \cite{Osawa-22} that the two-dimensional Zakharov--Kuznetsov equation is also locally well-posed in $H^s(\mathbb{R}\times\mathbb{T})$, for $ 9/10 < s < 1$.

By contrast, the mathematical theory pertaining to the study of controllability properties is considerably less advanced, especially in what concerns the two-dimensional linear dispersive-type equations on a periodic setting. To our knowledge, the first result on the subject was obtained by Rivas and Sun in \cite{RiSu-20}, where the authors studied the internal controllability of the Kadomtsev--Petviashvili II equation. More precisely, by applying the Hilbert Uniqueness Method due to Lions \cite{Lions-1988} and the techniques of semiclassical and microlocal analysis, they proved the controllability in $L^2(\mathbb{T}\times\mathbb{T})$ from a vertical strip. Additionally, a negative result for the controllability in $L^2(\mathbb{T}\times\mathbb{T})$ from a horizontal strip is also shown.

Recently, Pastor and Vielma Leal \cite{LePa-22} applied the moment method \cite{Russell-78} to establish the controllability of several two-dimensional linear dispersive equations posed on the two-dimensional torus in the framework of Sobolev spaces $H^s(\mathbb{T}\times\mathbb{T})$ with $s\geq 0$. To obtain the results, the author assumes that the control has a specific form and acts on the union of a narrow vertical and a narrow horizontal strip. As an application of the controllability results, they built an appropriate linear bounded operator to show that the resulting closed-loop system is exponentially stable with an arbitrary exponential decay rate. We emphasize that this is a direct consequence of the Hilbert Uniqueness Method mentioned above and the classical principle: Exact controllability implies exponential stabilizability for conservative control systems \cite{Liu-1997,Slemrod-1974}. Their analysis also covers the linearized equation \eqref{11}. In this context, it is worth mentioning the work \cite{CaGoMu} where Capistrano-Filho et al investigated the nonlinear Kawahara--Kadomtsev--Petviashvili II equation with a damping mechanism $a(x, y)u$, acting as a feedback-damping mechanism, and an anti-damping mechanism, preventing energy decrease. In addition, Rosier and Mo Chen addressed in \cite{MoRo-20} the exact controllability of the linear equation \eqref{11} on a rectangle with a left Dirichlet boundary control by using the flatness approach.

\subsection{Problem setting}
In this paper, we are mainly concerned with the internal observability,  controllability, and rapid stabilization of the linear ZK equation with $2\pi$-periodic boundary condition.
More precisely, given a subset $\Omega\subset \TT\times\TT$ and a Hilbert space $\mathcal{H}$ of initial data, we study the following questions:
\vspace{0.2cm}

\textbf{Problem $\mathcal{A}$:}
Given $T>0$, is it possible to recover, at least theoretically, the initial data $u_0\in\mathcal{H}$ from the observation of the solution of the homogeneous system
\begin{equation}\label{13}
\begin{cases}
&u_t+u_x+u_{xxx}+u_{xyy}=0\qtq{in}\TT\times\TT\times\RR,\\
&u(0)=u_0\qtq{in}\TT\times\TT
\end{cases}
\end{equation}
on $\Omega\times (0,T)$?
\vspace{0.2cm}

Here and in the sequel we often consider $u$ as a function of $t$ with values in a space of functions of $(x,y)$, and we write $u(t)$  instead of $u(\cdot,\cdot,t)$ for brevity.

\vspace{0.2cm}
\textbf{Problem $\mathcal{B}$:}
Given $T>0$ and $z_0, z_1\in\mathcal{H}$, can we find an appropriate control input $v$ supported on $\Omega\times (0,T)$ such that the corresponding solution $z$ of the system
\begin{equation}\label{14}
\begin{cases}
&z_t+z_x+z_{xxx}+z_{xyy}=v\qtq{in}\TT\times\TT\times\RR,\\
&z(0)=z_0\qtq{in}\TT\times\TT
\end{cases}
\end{equation}
satisfies the final condition $z(T)=z_1$?
\vspace{0.2cm}

\textbf{Problem $\mathcal{C}$:}
Given $\omega>0$ (arbitrarily large), can we construct a feedback law $v=F_{\omega}z$ such that the solutions of the resulting closed-loop system \eqref{14} decays exponentially to zero at least as fast as the function $e^{-\omega t}$?
\vspace{0.2cm}

We emphasize that we are interested in finding ``small'' observation and control domains $\Omega$. Problem $\mathcal{A}$ will be answered by employing diverse tools from nonharmonic analysis and Diophantine approximations, and specific geometric constructions. Based on the observability results, Problem $\mathcal{B}$ will be solved by applying the observation-control duality theory of Dolecki and Russell \cite{DolRus-1977} in the set-up of Lions \cite{Lions-1988}. Finally, Problem $\mathcal{C}$ will be solved by applying a variant of the Hilbert Uniqueness Method, developed in \cite{Kom-1997}.

\subsection{Notation and main results}
To state our results we introduce some notations. Given any two real numbers $r$ and $s$, we define the Hilbert space $H^{r,s}(\TT\times\TT)$ as the completion of the vector space spanned by the functions
\begin{equation*}
(x,y)\mapsto e^{i(mx+ny)}
\end{equation*}
with respect to the Euclidean norm
\begin{equation*}
\norm{\sum_{m,n\in\ZZ}c_{m, n}e^{i(mx+ny)}}_{r,s}
:=\bj{\sum_{m,n\in\ZZ}\bj{(1+m^2)^r+(1+n^2)^s}\abs{c_{m, n}}^2}^{1/2}.
\end{equation*}
Note that $L^2(\TT\times\TT)=H^{0,0}(\TT\times\TT)$.

We will also need a special subspace of $H^{r,s}(\TT\times\TT)$, namely the orthogonal complement of the functions $u(x,y)$ depending only on the variable $y$. It will be denoted by $H^{r,s}_x(\TT\times\TT)$, and we write $L^2_x(\TT\times\TT)$ instead of $H^{0,0}_x(\TT\times\TT)$. Finally, as usual, we will identify the $L^2$ spaces occurring in this paper with their duals.

We start by investigating the observability of the homogeneous system \eqref{13}. Our first result states that the initial data may be uniquely determined by the observation of the solutions on any fixed vertical segment during an arbitrarily small time interval (see the figure below). More precisely, we have the following result:

\begin{theorem}\label{t11}
For any fixed $x_0\in\TT$ and bounded interval $I$ of length $|I|>0$, there exist two  constants $C_1, C_2>0$ such that all solutions of \eqref{13} with $u_0\in L^2(\TT\times\TT)$  satisfy the following estimates:
\begin{equation}\label{15}
C_1\norm{u_0}_{L^2(\TT\times\TT)}^2 
\le\int_I\int_{\TT}\abs{u(x_0,y,t)}^2\ dy\ dt
\le C_2\norm{u_0}_{L^2(\TT\times\TT)}^2. 
\end{equation}
\end{theorem}

We point out that Theorem \ref{t11} is not valid for horizontal segments instead of vertical ones, but we have positive results by observing the solutions on \emph{at least two} horizontal segments  (see the figure below), and by considering special initial data {belonging to $L^2_x(\TT\times\TT)$}. 

If we choose initial data that are orthogonal to the set of functions depending only on $x$

The result, in this context, is the following one:

\begin{theorem}\label{t12}
We consider the solutions of the equation \eqref{13}. 
Let $N\ge 2$ be an integer, $y_1, \ldots, y_N\in\TT$, and $I_1, \ldots, I_N$ bounded intervals  of positive lengths.
\begin{enumerate}[\upshape (i)]
\item If $N=2$ and $(y_1-y_2)/\pi$ is irrational,
then the linear map
\begin{equation*}
u_0\mapsto \bj{u_{\TT\times\set{y_1}\times I_1}, u_{\TT\times\set{y_2}\times I_2}}
\end{equation*}
is one-to-one on $L^2_x(\TT\times\TT)$.

\item If {$y_1, \ldots, y_{N+1}$ and $ \pi$ } are linearly independent over the field of rational numbers and  they belong to a real algebraic field  of degree $N+2$, then there exists a  constant $C_1>0$  such that the {solution of \eqref{11} with every initial value } $u_0\in L^2_x(\TT\times\TT)$ satisfy the following estimate:
\begin{equation*}
C_1\norm{u_0}_{H^{0,-1/N}(\TT\times\TT)}^2\le \sum_{j=1}^{N+1}\int_{I_j}\int_{\TT}
\norm{u(x,y_j,t)}^2\, dx\, dt.
\end{equation*}
\end{enumerate}
\end{theorem}

\begin{remark}\label{r13}
In case $N=1$, Theorem \ref{t12} (ii) states that if $(y_1-y_2)/\pi$ is a quadratic irrational number, then
\begin{equation*}
C_1\norm{u_0}_{H^{0,-1}(\TT\times\TT)}^2\le \sum_{j=1}^2\int_{I_j}\int_{\TT}
\norm{u(x,y_j,t)}^2\, dx\, dt
\end{equation*}
with some constant $C_1>0.$ Our proof will show that a slightly weaker estimate holds under the much weaker assumption that $(y_1-y_2)/\pi$ is an arbitrary irrational algebraic number:
for each real number $s>1$ there exists a constant $C_s>0$ such that
\begin{equation*}
C_s\norm{u_0}_{H^{0,-s}(\TT\times\TT)}^2\le \sum_{j=1}^2\int_{I_j}\int_{\TT}
\norm{u(x,y_j,t)}^2\, dx\, dt.
\end{equation*}
\end{remark}

The next observability theorem states that the initial data may also be uniquely determined by the observation of the solutions on an arbitrarily small disc in $\TT\times\TT$ during an arbitrarily small time interval.

\begin{theorem}\label{t14}
For every fixed nonempty open set $U\subset\TT\times\TT\times\RR$, there exist two  constants $C_1, C_2>0$ such that the solutions of \eqref{13} with $u_0\in L^2_x(\TT\times\TT)$ satisfy the following estimates:
\begin{equation*}
C_1\norm{u_0}_{L^2(\TT\times\TT)}^2
\le\int_{U}\abs{u(x,y,t)}^2\ dx\ dy\ dt
\le C_2\norm{u_0}_{L^2(\TT\times\TT)}^2.
\end{equation*}
\end{theorem}

For a better understanding of the observation regions considered in the previous theorems, we present three figures below:
\begin{figure}[h]
\centering
\begin{minipage}[b]{0.25\textwidth}
\scalebox{1}{\begin{tikzpicture}
\draw[thick] (0,0)--(0,4)--(4,4)--(4,0)--(0,0);
\draw[ultra thick,red] (1,0)--(1,4);
\end{tikzpicture}}
\end{minipage}%
\hspace{0.06\textwidth}
\begin{minipage}[b]{0.25\textwidth}
\scalebox{1}{\begin{tikzpicture}
\draw[thick] (0,0)--(0,4)--(4,4)--(4,0)--(0,0);
\draw[ultra thick,red] (0,1)--(4,1);
\draw[ultra thick,red] (0,2)--(4,2);
\end{tikzpicture}}
\end{minipage}%
\hspace{0.05\textwidth}
\begin{minipage}[b]{0.25\textwidth}
\scalebox{1}{\begin{tikzpicture}
\draw[thick] (0,0)--(0,4)--(4,4)--(4,0)--(0,0);
\draw[fill=red] (2,2) circle (0.5);
\end{tikzpicture}}
\end{minipage}
\caption*{Theorems \ref{t11}, \ref{t12} and \ref{t14}}
\end{figure}
\begin{remark}
Theorems \ref{t11}, \ref{t12}, and \ref{t14}, improve the results obtained in \cite{LePa-22} for the system \eqref{11}. The improvements
are given to the observation regions. Indeed, as remarked above, the results in \cite{LePa-22} were proved by assuming that the controls act on the union of a narrow vertical and a narrow horizontal strip.
\end{remark}

Next, we consider the nonhomogeneous problem \eqref{14}.  We will prove the following exact controllability and rapid stabilization theorems:

\begin{theorem}\label{t15}
Fix $x_0\in\TT$ and $T>0$ arbitrarily.
Given any $z_0, z_T\in L^2(\TT\times\TT)$, there exists a  function $v\in L^2(\set{x_0}\times\TT\times (0,T))$ such that the solution of the control problem \eqref{14} satisfies the final condition
\begin{equation*}
z(T)=z_T\qtq{in}\TT\times\TT.
\end{equation*}
\end{theorem}

\begin{theorem}\label{t16}
Fix $x_0\in\TT$ and a number $\omega>0$ arbitrarily.
There exists a linear map
\begin{equation*}
F_{\omega}:L^2(\TT\times\TT)\to L^2(\set{x_0}\times\TT)
\end{equation*}
and a constant $C>0$ such that
the  closed-loop problem
\begin{equation}\label{16}
\begin{cases}
&z_t+z_x+z_{xxx}+z_{xyy}=F_{\omega}z\qtq{in}\TT\times\TT\times (0,+\infty),\\
&z(0)=z_0\qtq{in}\TT\times\TT,
\end{cases}
\end{equation}
is well-posed in $L^2(\TT\times\TT)$, and its solutions satisfy the estimate
\begin{equation*}
\norm{z(t)}_{L^2(\TT\times\TT)} \leq C\norm{z_0}_{L^2(\TT\times\TT)} e^{-\omega t},
\end{equation*}
for all $z_0 \in L^2(\TT\times\TT)$ and $t \geq 0$.
\end{theorem}

\begin{theorem}\label{t17}
Fix a nonempty open set $\Omega\subset \TT\times\TT$ and  $T>0$ arbitrarily.
Given any $z_0, z_T\in \bj{L^2_x(\TT\times\TT)}'$, there exists a  function $v\in L^2(\Omega\times (0,T))$ such that the solution of the control problem \eqref{14} satisfies the final condition
\begin{equation*}
z(T)=z_T\qtq{in}\TT\times\TT.
\end{equation*}
\end{theorem}

\begin{theorem}\label{t18}
Fix a nonempty open set
$\Omega\subset\TT\times\TT$ and a number $\omega>0$ arbitrarily.
There exists a linear map
\begin{equation*}
F_{\omega}:\bj{L^2_x(\TT\times\TT)}'\to L^2(\Omega)
\end{equation*}
and a constant $C>0$ such that
the  closed-loop problem \eqref{16} is well-posed in $\bj{L^2_x(\TT\times\TT)}'$, and its solutions satisfy the estimate
\begin{equation*}
\norm{z(t)}_{L^2(\TT\times\TT)} \leq C\norm{z_0}_{L^2(\TT\times\TT)} e^{-\omega t},
\end{equation*}
for all $z_0 \in \bj{L^2_x(\TT\times\TT)}'$ and $t \geq 0$.
\end{theorem}

We finish this subsection by remarking that we could also obtain exact controllability and rapid stabilization theorems corresponding to Theorem \ref{t12}, but the appropriate Hilbert spaces are not Sobolev spaces. The treatment is similar to that of the wave equation with the Neumann boundary condition as in \cite{Lions-1988} and \cite{Kom-1997}, respectively.

\subsection{Structure of the article}

Our work is composed of nine sections, including the Introduction. In Section \ref{s2} we briefly discuss the well-posedness of the homogeneous system \eqref{14}, and the eigenfunction expansion of its solutions. The special form of its spectrum will play a crucial role in all subsequent results. In Section \ref{s3} we give a short review of some classical theorems of Ingham, Beurling, and Kahane. We introduce the notion of a \emph{sparse set}; this will simplify our applications to observability. The most important mathematical proofs of the present paper are given in Sections \ref{s4} and \ref{s5}, where we prove Theorems \ref{t11}, \ref{t12}, and \ref{t14}.  Section \ref{s6} is devoted to recalling two theorems of the kind
\begin{equation*}
\text{observability}\Longleftrightarrow\text{controllability}
\qtq{and} \text{observability}\Longrightarrow\text{rapid stabilizability}
\end{equation*}
in an abstract framework. In particular, we emphasize that the well-posedness of a linear control system may easily be derived from the observability of the dual system. Applying the abstract results of  Section \ref{s6}, we deduce Theorems \ref{t15}--\ref{t18} from the results of Sections \ref{s4} and \ref{s5}, in the Section \ref{s7}. In our final section \ref{s8}, we indicate further possible applications of the methods developed in this manuscript, and we formulate some related open problems.

\section{Well-posedness and eigenfunction expansions}\label{s2}

We establish the well-posedness of the homogeneous system \eqref{13}. We have the following result:

\begin{proposition}\label{p21}
Let $s\in\RR$.
Given any initial datum $u_0\in H^s(\TT\times\TT)$, {the problem \eqref{13} } has a unique solution
\begin{equation*}
u\in C(\RR;H^s(\TT\times\TT))\cap C^1(\RR;H^{s-3}(\TT\times\TT)).
\end{equation*}
Furthermore, it is represented by the series
\begin{equation}\label{21}
u(x,y,t)=\sum_{m,n\in\ZZ}c_{m, n}e^{i(mx+ny+\omega_{m, n}t)},\quad \omega_{m, n}:=m^3+mn^2-m
\end{equation}
with suitable complex coefficients $c_{m, n}$, depending on $u_0$.

Moreover, the map $u_{S}\mapsto u_0$ is an isometric automorphism of $L^2(\TT\times\TT)$ for every $S\in\RR$, and
\begin{equation*}
\norm{u(S)}_{L^2(\TT\times\TT)}^2
=4\pi^2\sum_{m,n\in\ZZ}\abs{c_{m, n}}^2.
\end{equation*}
\end{proposition}

\begin{proof}
The proof follows a standard application of linear semigroup theory.  The present problem is even simpler because there is an orthonormal base formed of eigenfunctions of the underlying operator
\begin{equation*}
\mathcal{A}u:=-(u_x+u_{xxx}+u_{xyy}),\quad u\in H^{s+3}(\TT\times\TT).
\end{equation*}

The first part of the theorem readily follows for example by applying \cite[Theorem 3.1]{KomLor-2005} for this operator in the Hilbert space $\mathcal{H}=H^s(\TT\times\TT)$.

The isometry property is a consequence of Parseval's formula:
\begin{align*}
\norm{u(S)}_{L^2(\TT\times\TT)}^2
&=\int_{\TT\times\TT}\abs{\sum_{m,n\in\ZZ}c_{m, n}e^{i(mx+ny+\omega_{m, n}S)}}^2\, dx\, dy\\
&=4\pi^2\sum_{m,n\in\ZZ}\abs{c_{m, n}e^{i\omega_{m, n}S}}^2
=4\pi^2\sum_{m,n\in\ZZ}\abs{c_{m, n}}^2
\end{align*}
for every $S\in\RR$.

Finally, each map $u_{S}\mapsto u_0$ {is onto } because if $(c_{m,n})_{(m,n)\in\ZZ^2}$ runs over $\ell^2(\ZZ^2)$, then $(c_{m,n}e^{i\omega_{m,n}S})_{(m,n)\in\ZZ^2}$ also runs over $\ell^2(\ZZ^2)$.
\end{proof}

\section{A short review of Ingham type theorems}\label{s3}

First, we recall a theorem of Ingham \cite{Ingham-1936} and its strengthened form given by Beurling \cite{Beurling} and Kahane \cite{Kahane-1962}.

\begin{theorem}\label{t31}
Let $(\lambda_k)_{k\in K}$ be a uniformly separated family of real numbers, i.e., satisfying the following \emph{gap condition}:
\begin{equation*}
\gamma:=\inf_{\substack{m,n\in K\\ k\ne n}}\abs{\lambda_k-\lambda_n}>0.
\end{equation*}
Then for each bounded interval $I$ of length $\abs{I}>\frac{2\pi}{\gamma}$ there exist two  constants $C_1, C_2>0$ such that every square summable family $(c_k)_{k\in K}$ of complex numbers satisfies the following inequalities:
\begin{equation}\label{31}
C_1\sum_{k\in K}\abs{x_k}^2
\le\int_I\abs{\sum_{k\in K}c_ke^{i\lambda_kt}}^2\, dt
\le C_2\sum_{k\in K}\abs{x_k}^2.
\end{equation}

Moreover, the conclusion holds (with different constants) under the weaker condition
\begin{equation*}
\abs{I}>\frac{2\pi}{{\gamma(K\setminus L)}}\qtq{with}
{\gamma(K\setminus L)}:=\inf_{\substack{m,n\in K\setminus L\\ k\ne n}}\abs{\lambda_k-\lambda_n},
\end{equation*}
{where $L$ is any finite subset of $K$.}
\end{theorem}

\begin{example}
The family $(\lambda_k:=k^3)_{k\in\ZZ}$ satisfies the assumptions of the theorem with $\gamma=1$.
Furthermore, if $L_m=\set{k\in\ZZ\ :\ |k|<m}$ for some  integer $m\ge 1$, then
\begin{equation*}
{\gamma(K\setminus L_m)=3m(m-1)+1 }
\end{equation*}  
by a simple verification.
Since $\gamma(K\setminus L_m)\to +\infty$ as $m\to+\infty$, the inequalities \eqref{31} (with $K=\ZZ$) hold for  bounded intervals of arbitrary positive length.
\end{example}

Theorem \ref{t31} was extended to higher dimensions by Kahane \cite{Kahane-1962}.
We only state here a special case of his results that will be sufficient for our applications. To do that, let us introduce some notations.

\begin{definition}
Consider a set $\Lambda=\set{\lambda_k\ :\ k\in K}$ in a normed space.

$\Lambda$ is \emph{uniformly separated} if
\begin{equation*}
\gamma(\Lambda):=\inf_{\substack{m,n\in K\\ k\ne n}}\norm{\lambda_k-\lambda_n}>0.
\end{equation*}

$\Lambda$ is \emph{sparse} if for every $\eps>0$ there exists a finite cover $\Lambda=\Lambda_1\cup\cdots\cup \Lambda_m$ of $K$ such that
\begin{equation*}
\frac{1}{\gamma(\Lambda_1)}+\cdots+\frac{1}{\gamma(\Lambda_m)}<\eps.
\end{equation*}
\end{definition}

\begin{remarks}\mbox{}
\begin{enumerate}[\upshape (i)]
\item It follows from the definition that both properties are invariant for similarities, and they are {preserved by taking subsets.}
\item The empty set is sparse.
The singletons $\set{\lambda}$ are sparse because $\gamma(\set{\lambda})=+\infty$ by definition.

\item The union of finitely many sparse sets $\Lambda_1$, \ldots, $\Lambda_m$ is still sparse.
Indeed, given $\eps>0$, it suffices to apply the definition to each $\Lambda_j$ with $\eps/m$ instead of $\eps$ and then to take the union of all these covers.

\item Every finite set is sparse by (ii) and (iii).

{\item A sparse set remains sparse in every other equivalent norm.}
\end{enumerate}
\end{remarks}

Let us give a typical nontrivial example:

\begin{example}\label{e34}
The set $\Lambda=\set{k^3\ :\ k\in\ZZ}$ is sparse in $\RR$.
Indeed, for any fixed $\eps>0$ choose an integer $m>1/3\eps$, and consider the finite cover $\Lambda_1, \ldots, \Lambda_{2m+1}$ of $\Lambda$ by the  sets
\begin{equation*}
\set{k^3\ :\ k\ge m},\quad
\set{k^3\ :\ k\le -m}
\end{equation*}
and the singletons $\set{j^3}$ for $j=-m+1, \ldots, m-1$.
The first set has a uniform gap {$>3m^2$ } because if $k_1>k_2\ge m$, then
\begin{equation*}
k_1^3-k_2^3\ge (k_2+1)^3-k_2^3=3k_2^2+3k_2+1>{3m^2.}
\end{equation*}
The second set is a reflection of $\Lambda_1$, hence it has the same uniform gap.
Since the singletons have infinite uniform gaps, we conclude that
\begin{equation*}
\sum_{j=1}^{2m+1}\frac{1}{\gamma(\Lambda_j)}
<\frac{1}{{3m^2}}+\frac{1}{{3m^2}}+\frac{1}{+\infty}+\cdots+\frac{1}{+\infty}
={\frac{2}{3m^2}}<\eps.
\end{equation*}
\end{example}

The following theorem shows the importance of sparse sets; is an easy corollary of a more general theorem of Kahane  \cite{Kahane-1962} (see also \cite{Baiocchi-Komornik-Loreti-1999} and \cite[Theorem 8.1 and Proposition 8.4]{KomLor-2005} for some improvements).

\begin{theorem}\label{t35}
Let $(\lambda_k)_{k\in K}$ be a uniformly separated and sparse family of vectors in $\RR^N$. Then for each nonempty bounded open set $G\subset\RR^N$ there exist two  constants $C_1, C_2>0$ such that
\begin{equation*}
C_1\sum_{k\in K}\abs{x_k}^2
\le\int_G\abs{\sum_{k\in K}c_ke^{i\lambda_k\cdot z}}^2\, dz
\le C_2\sum_{k\in K}\abs{x_k}^2,
\end{equation*}
for every square summable family $(x_k)_{k\in K}$ of complex numbers.
\end{theorem}

\section{Observability on segments}\label{s4}
This section is dedicated to giving the proof of Theorems \ref{t11} and \ref{t12}. Let us start with the proof of Theorem \ref{t11}. We will apply Ingham's inequality infinitely many times, by ensuring the existence of the same constants for each application. This kind of proof was first applied in \cite{Haraux-1989}.

\begin{proof}[Proof of Theorem \ref{t11}]
Fix an integer $N\ge 1$ such that $\frac{2\pi}{N^2}<|I|$.
For each fixed $n$ with $\abs{n}\ge N$ the sequence
\begin{equation*}
(\omega_{m,n})_{m\in\ZZ}=(m^3+mn^2-m)_{m\in\ZZ}
\end{equation*}
has a uniform gap $\ge N^2$ because
\begin{equation*}
(m+1)^3+(m+1)n^2-(m+1)-(m^3+mn^2-m)
=3m(m+1)+n^2
\ge n^2\ge N^2,
\end{equation*}
for every $m\in\ZZ$.

It follows that we may apply Ingham's theorem for each of these sequences with  common uniform constants $C_1', C_2'>0$, depending only on $|I|$:
\begin{equation}\label{41}
C_1'\sum_{m\in\ZZ}\abs{c_{m,n}}^2
\le\int_I\abs{\sum_{m\in\ZZ}c_{m, n}e^{i(mx_0+\omega_{m, n}t)}}^2\, dt
\le C_2'\sum_{m\in\ZZ}\abs{c_{m,n}}^2,
\end{equation}
for each fixed $n\in\ZZ\setminus\set{-N+1, \ldots, N-1}$, in the previous inequality we have used the following trivial equality $\abs{c_{m,n}e^{imx_0}}=\abs{c_{m,n}}$.

For each of the remaining $2N-1$ values of $n$, the truncated sequence
\begin{equation*}
(m^3+mn^2-m)_{\abs{m}\ge N+1}
\end{equation*}
also has a uniform gap $\ge N^2$ because
\begin{equation*}
3m(m+1)+n^2=3|m|\cdot |m+1|+n^2\ge N^2,
\end{equation*}
whenever $\abs{m}\ge N+1$ and thus $\abs{m+1}\ge N$.
Applying the last part of Theorem \ref{t31}, there exist suitable constants $C_{1,n}', C_{2,n}'>0$, depending only on $|I|$, such that
\begin{equation}\label{42}
C_{1,n}'\sum_{m\in\ZZ}\abs{c_{m,n}}^2
\le\int_I\abs{\sum_{m\in\ZZ}c_{m, n}e^{i(mx_0+\omega_{m, n}t)}}^2\, dt
\le C_{2,n}'\sum_{m\in\ZZ}\abs{c_{m,n}}^2,
\end{equation}
for each fixed $n\in\set{-N+1, \ldots, N-1}$.

Setting
\begin{equation*}
C_1'':=\min\set{C_1',C_{1,1-N}',\ldots, C_{1,N-1}'}
\qtq{and}
C_2'':=\max\set{C_2',C_{2,1-N}',\ldots, C_{2,N-1}'},
\end{equation*}
and, adding the equations \eqref{41} and \eqref{42}, we obtain that
\begin{equation*}
C_1''\sum_{m\in\ZZ}\abs{c_{m,n}}^2
\le\int_I\abs{\sum_{m\in\ZZ}c_{m, n}e^{i(mx_0+\omega_{m, n}t)}}^2\, dt
\le C_2''\sum_{m\in\ZZ}\abs{c_{m,n}}^2,
\end{equation*}
for each $n\in\ZZ$.
Summing them we obtain the following estimate:
\begin{equation}\label{43}
C_1''\sum_{m,n\in\ZZ}\abs{c_{m,n}}^2
\le \sum_{n\in\ZZ}\int_I\abs{\sum_{m\in\ZZ}c_{m, n}e^{i(mx_0+\omega_{m, n}t)}}^2\, dt
\le C_2''\sum_{m,n\in\ZZ}\abs{c_{m,n}}^2.
\end{equation}
Since
\begin{align*}
\int_I\int_{\TT}\abs{u(x_0,y,t)}^2\ dy\ dt
&=\int_I\int_{\TT}\abs{\sum_{n\in\ZZ}e^{iny}\sum_{m\in\ZZ}c_{m, n}e^{i(mx_0+\omega_{m, n}t)}}^2\ dy\ dt\\
&=2\pi \sum_{n\in\ZZ}\int_I\abs{\sum_{m\in\ZZ}c_{m, n}e^{i(mx_0+\omega_{m, n}t)}}^2\ dt
\end{align*}
by Parseval's formula, and
\begin{equation*}
\norm{u_0}_{L^2(\TT\times\TT)}^2 =4\pi^2\sum_{m,n\in\ZZ}\abs{c_{m, n}}^2
\end{equation*}
by Proposition \ref{p21}, the estimates \eqref{15} of Theorem \ref{t11} follow from \eqref{43} with $C_1=C_1''/2\pi$ and $C_2=C_2''/2\pi$.
\end{proof}

Now we turn to the proof of Theorem \ref{t12}.
{Since we consider only initial data belonging to $L^2_x(\TT\times\TT)$, all coefficients $c_{0,n}$ vanish in the series representation \eqref{21} of the solutions. 
Since  $\omega_{m,n}=\omega_{m,-n}$, we have for any fixed $y_0\in\TT$ and for every non-degenerate bounded open interval $I$ }the following identity by a similar computation  as in the preceding proof:
\begin{align*}
\int_I\int_{\TT}&\abs{u(x,y_0,t)}^2\, dx\, dt\\
&=\int_I\int_{\TT}\abs{\sum_{{m\in\ZZ^*}}e^{imx}\sum_{n\in\ZZ}c_{m, n}e^{i(ny_0+\omega_{m, n}t)}}^2\, dx\, dt\\
&=2\pi \sum_{{m\in\ZZ^*}}\int_I\abs{\sum_{n\in\ZZ}c_{m, n}e^{i(ny_0+\omega_{m, n}t)}}^2\, dt\\
&=2\pi \sum_{{m\in\ZZ^*}}
\int_I\abs{c_{m, 0}{ e^{i\omega_{m, 0}t}}
+\sum_{n=1}^{\infty}\bj{c_{m, n}e^{iny_0}+c_{m, -n}e^{-iny_0}}e^{i\omega_{m, n}t}}^2\, dt.
\end{align*}
{Applying Ingham's theorem to the last expression as in the proof of Theorem \ref{t11}, we obtain   the equivalence
\begin{equation*}
\int_I\int_{\TT}\abs{u(x,y_0,t)}^2\, dx\, dt
\asymp\sum_{m\in\ZZ^*}\bj{\abs{c_{m, 0}}^2
+\sum_{n=1}^{\infty}\abs{c_{m, n}e^{iny_0}+c_{m, -n}e^{-iny_0}}^2}.
\end{equation*}
This shows that the observation on one horizontal segment does not allow us to determine separately $c_{m, n}$ and $c_{m, -n}$, but only a linear combination of them.}

If we observe the solution on two horizontal segments on identical or different time intervals $I_1$ and $I_2$, then we obtain that
\begin{equation}\label{44}
\int_{I_1}\int_{\TT}\abs{u(x,y_1,t)}^2\, dx\, dt
+\int_{I_2}\int_{\TT}\abs{u(x,y_2,t)}^2\, dx\, dt
\end{equation}
is equivalent to
\begin{equation*}
\sum_{m\in\ZZ^*}\bj{\abs{c_{m, 0}}^2
+\sum_{n=1}^{\infty}\bj{\abs{c_{m, n}e^{iny_1}+c_{m, -n}e^{-iny_1}}^2+\abs{c_{m, n}e^{iny_2}+c_{m, -n}e^{-iny_2}}^2}},
\end{equation*}
or to
\begin{equation}\label{45}
\sum_{m\in\ZZ^*}\bj{\abs{c_{m, 0}}^2
+\sum_{n=1}^{\infty}\bj{\abs{c_{m, n}e^{2iny_1}+c_{m, -n}}^2+\abs{c_{m, n}e^{2iny_2}+c_{m, -n}}^2}}.
\end{equation}
With this in hand, we are now in a position to prove the second main result of this article.

\begin{proof}[Proof of Theorem \ref{t12}]
(i) In view of the equivalence between \eqref{44} and \eqref{45} it suffices to show the implications
\begin{equation*}
c_{m, n}e^{2iny_1}+c_{m, -n}=c_{m, n}e^{2iny_2}+c_{m, -n}=0
\Longrightarrow
c_{m, n}=c_{m, -n}=0.
\end{equation*}
This holds if and only if $e^{2iny_1}\ne e^{2iny_2}$, that is, if $(y_1-y_2)/\pi$ is irrational, showing item (i).

\vspace{0.2cm}

(ii) Let us first assume that $N=1$, i.e., we observe the solution on two segments.
Then we have to prove the following estimate with some constant $\gamma>0$:
\begin{equation}\label{46}
\int_{I_1}\int_{\TT}\abs{u(x,y_1,t)}^2\, dx\, dt
+\int_{I_2}\int_{\TT}\abs{u(x,y_2,t)}^2\, dx\, dt
\ge \gamma\sum_{m\in\ZZ^*}\sum_{n\in\ZZ}\frac{\abs{c_{m, n}}^2}{n^2}.
\end{equation}

{Since $(y_1-y_2)/\pi$ is irrational, $e^{2iny_1}\ne e^{2iny_2}$ for every positive integer $n$.
Therefore we may express $c_{m, n}$ and $c_{m, -n}$ by
\begin{equation*}
a_{m, n}:=c_{m, n}e^{2iny_1}+c_{m, -n}
\qtq{and}
b_{m, n}:=c_{m, n}e^{2iny_2}+c_{m, -n}.
\end{equation*}
We obtain the equalities
\begin{equation*}
c_{m, n}=\frac{a_{m, n}-b_{m, n}}{e^{2iny_1}-e^{2iny_2}}
\qtq{and}
c_{m, -n}=\frac{b_{m, n}e^{2iny_1}-a_{m, n}e^{2iny_2}}{e^{2iny_1}-e^{2iny_2}},
\end{equation*}
and they imply the following estimates:
\begin{equation*}
\abs{c_{m, n}}^2+\abs{c_{m, -n}}^2
\le 4\frac{\abs{a_{m, n}}^2+\abs{b_{m, -n}}^2}{\abs{e^{2iny_1}-e^{2iny_2}}^2}.
\end{equation*}
We may write it in the following equivalent form:
\begin{equation*}
\abs{c_{m, n}e^{2iny_1}+c_{m, -n}}^2+\abs{c_{m, n}e^{2iny_2}+c_{m, -n}}^2
\ge \frac{\abs{e^{2in(y_1-y_2)}-1}^2}{4}
\bj{\abs{c_{m, n}}^2+\abs{c_{m, -n}}^2}.
\end{equation*}
Thanks to the fact that
\begin{equation*}
\abs{e^{2\pi i \theta}-1}\ge 2\gamma_1\dist (\theta,\ZZ),
\end{equation*}
}
for all real numbers $\theta$ with some constant $\gamma_1>0$, this yields 
\begin{multline*}
\abs{c_{m, n}e^{2iny_1}+c_{m, -n}}^2+\abs{c_{m, n}e^{2iny_2}+c_{m, -n}}^2\\
\ge \gamma_1^2 \bj{\dist \bj{\frac{n(y_1-y_2)}{\pi},\ZZ}}^2\bj{\abs{c_{m, n}}^2+\abs{c_{m, -n}}^2}.
\end{multline*}
Now we recall, e.g., from \cite[Case $n=1$ of Theorem III on p. 79]{Cassels-1957} that if $\theta$ is a quadratic irrational algebraic number, then there exists a constant $\gamma_2>0$ such that
\begin{equation*}
\dist (n\theta,\ZZ)\ge \frac{\gamma_2}{n},
\end{equation*}
for every integer $n\ge 1$.
Applying this with $\theta=(y_1-y_2)/\pi$ we obtain the  following estimate:
\begin{equation*}
\abs{c_{m, n}e^{2iny_1}+c_{m, -n}}^2+\abs{c_{m, n}e^{2iny_2}+c_{m, -n}}^2
\ge \frac{\gamma_1^2\gamma_2^2}{n^2} \bj{\abs{c_{m, n}}^2+\abs{c_{m, -n}}^2}.
\end{equation*}

Using this inequality the required estimate \eqref{46} follows from the equivalence of the expressions \eqref{44} and \eqref{45}.
\medskip

Under the weaker assumption that $(y_1-y_2)/\pi$ is an irrational algebraic number, we may repeat the preceding proof, by using at the end Roth's theorem (see. e.g., \cite[Theorem I on p. 104]{Cassels-1957}):
\begin{equation*}
\dist (n\theta,\ZZ)\ge \frac{\gamma_s}{n^s},
\end{equation*}
for every  integer $n\ge 1$.
This yields the theorem mentioned in Remark \ref{r13}.
\medskip

Now let us turn to the case $N\ge 2$.
Applying the first part of the proof for $N=1$ with $y_1-y_j$ for $j=2,\ldots, N+1$ we obtain the inequality
\begin{align*}
\sum_{j=1}^{N+1}\int_{I_j}\int_{\TT}
\norm{u(x,y_j,t)}^2\, dx\, dt
&\ge \gamma_1^2 \sum_{j=2}^{N+1}\bj{\dist (n(y_1-y_j)/\pi,\ZZ)}^2\bj{\abs{c_{m, n}}^2+\abs{c_{m, -n}}^2}\\
&\ge \gamma_1^2 \max_{j=2}^{N+1}\set{\dist (n(y_1-y_j)/\pi,\ZZ)}^2\bj{\abs{c_{m, n}}^2+\abs{c_{m, -n}}^2}.
\end{align*}
Thanks to our assumptions on $y_1,\ldots, y_{N+1}$ we may apply \cite[Theorem III on p. 79]{Cassels-1957} to obtain the inequality
\begin{equation*}
\max_{j=2}^{N+1}\set{\dist (n(y_1-y_j)/\pi,\ZZ)}
\ge \gamma_2n^{-1/N},
\end{equation*}
for all $n$.
Therefore we have
\begin{equation*}
\sum_{j=1}^{N+1}\int_{I_j}\int_{\TT}
\norm{u(x,y_j,t)}^2\, dx\, dt
\ge \gamma_1^2 \gamma_2^2n^{-2/N},
\end{equation*}
for all integers $n\ge 1$, so the proof of item (ii) is achieved, and Theorem \ref{t12} is proved.
\end{proof}

\section{Observability on small balls}\label{s5}
This section deals with the presentation of the proof of Theorem \ref{t14}. In this theorem, we consider only initial data belonging to  $L^2_x(\TT\times\TT):=H^{0,0}_x(\TT\times\TT)$.
Then the terms with $m=0$ are missing in formula \eqref{21}, so that the solution of \eqref{13} is given by
$$
u(x,y,t)=\sum_{(m,n)\in\ZZ^*\times\ZZ}c_{m, n}e^{i(mx+ny+\omega_{m, n}t)},\quad \omega_{m, n}:=m^3+mn^2-m,
$$
with square summable complex coefficients $c_{m, n}$.

In view of Theorem \ref{t35}, Theorem \ref{t14} will be proved if we establish the following  result:

\begin{proposition}\label{p51}
The set
\begin{equation*}
\Lambda=\set{\lambda_{m,n}:=(m,n,m^3+mn^2-m)\in\ZZ^3\ :\ (m,n)\in\ZZ^*\times\ZZ}
\end{equation*}
is uniformly separated and sparse.
\end{proposition}

For simplicity, we will consider the usual Euclidean norm of $\RR^N$; this is not essential because all other norms are equivalent to this one.
We recall from Remark 1 (v) that a sparse set remains sparse in every other equivalent norm.
Then $\Lambda$ is obviously uniformly separated with $\gamma(\Lambda)\ge 1$. We now prove that $\Lambda$ is sparse. To do that, we need four lemmas.

\begin{lemma}\label{l52}
A set $\Lambda$ in a normed space is sparse if it has the following property: for every $R>0$ there exists a sparse set $\Lambda'$ such that $\gamma(\Lambda\setminus \Lambda')\ge R$.
\end{lemma}

\begin{proof}
For any fixed $\eps>0$ choose $R>2/\eps$, and let $\Lambda_1$, \ldots, $\Lambda_m$ be a finite cover of the  corresponding sparse set $\Lambda'$ such that
\begin{equation*}
\frac{1}{\gamma(\Lambda_1)}+\cdots+\frac{1}{\gamma(\Lambda_m)}<\frac{\eps}{2}.
\end{equation*}
Then the finitely many sets $\Lambda_1, \ldots, \Lambda_m, \Lambda'$ cover of $\Lambda$, and
\begin{equation*}
\frac{1}{\gamma(\Lambda_1)}+\cdots+\frac{1}{\gamma(\Lambda_m)}+\frac{1}{\gamma(\Lambda\setminus \Lambda')}<\frac{\eps}{2}+\frac{1}{R}\le \eps.\qedhere
\end{equation*}
\end{proof}

The following is a generalization of Example \ref{e34}:

\begin{lemma}\label{l53}
Consider a set of the form
\begin{equation*}
\Lambda=\set{(g(m),f(m))\ :\ m\in\ZZ}\subset\RR^N,
\end{equation*}
where $g:\ZZ\to\ZZ^{N-1}$ is an arbitrary function, and $f:\ZZ\to\ZZ$ is a function satisfying the limit relation
\begin{equation}\label{52}
\lim_{|m|\to -\infty}\bj{f(m+1)-f(m)}
=+\infty.
\end{equation}
Then $\Lambda$ is sparse. The same conclusion holds if the limit is $-\infty$.
\end{lemma}

\begin{proof}
As mentioned before, we consider the Euclidean norm of $\RR^N$, so that
\begin{equation*}
\norm{(x_1,\ldots,x_N)}\ge\max\set{\abs{x_1},\ldots,\abs{x_N}}.
\end{equation*}

It follows from our assumptions that
\begin{equation*}
\lim_{m\to -\infty}f(m)=-\infty\qtq{and}
\lim_{m\to +\infty}f(m)=+\infty.
\end{equation*}
Therefore, for any fixed $R>0$ we may choose an  integer $n\ge 1$ such that
\begin{align*}
&f(n)-f(-n)\ge R,\\
&f(m+1)-f(m)\ge R\qtq{for all}m\le -n, \quad \text{and}\\
&f(m+1)-f(m)\ge R\qtq{for all}m\ge n.
\end{align*}
In particular, $f$ is increasing in the intervals $(-\infty,-n]$ and $[n,+\infty)$.
We complete the proof of the first part by showing that the set
\begin{equation*}
\Lambda':=\set{(g(m),f(m))\ :\ |m|<n}
\end{equation*}
satisfies the hypotheses of Lemma \ref{l52}.

Indeed, first of all, $\Lambda\setminus \Lambda'$ is a finite set, hence it is sparse.
It remains to show that if $|m_1|\ge n$, $|m_2|\ge n$ and $m_1\ne m_2$, then
\begin{equation*}
\norm{(g(m_1),f(m_1))-(g(m_2),f(m_2))}\ge R.
\end{equation*}
We may assume by symmetry that $m_1>m_2$, so that $m_1\ge m_2+1$.

We are going to prove the stronger inequality $f(m_1)-f(m_2)\ge R$.
It is satisfied if $m_1$ and $m_2$ have different signs, because then
\begin{equation*}
f(m_1)-f(m_2)\ge f(n)-f(-n)\ge R.
\end{equation*}
If $m_1$ and $m_2$ have equal signs, then
\begin{equation*}
f(m_1)-f(m_2)\ge f(m_2+1)-f(m_2)\ge R.
\end{equation*}

If the limit is $-\infty$ is the assumption of the lemma, then $-\Lambda$ is sparse by the first part, and then its reflection $\Lambda$ is also sparse.
\end{proof}

Based on Lemmas \ref{l52}, \ref{l53} and on the fact that a finite union of sparse sets is still sparse, our strategy for the proof of Proposition \ref{p51} is the following: for any given $R>0$ we will construct a finite number of sparse sets $\Lambda_{k,\ell,R}$ of the type studied in Lemma \ref{l53} such that their union $\Lambda'$ satisfies the uniform gap condition $\gamma(\Lambda\setminus \Lambda')\ge R$.

The construction of the sets $\Lambda_{k,\ell,R}$ is inspired by a method developed by Jaffard \cite{Jaffard-1990} for the proof of the interior observability of rectangular plates.

To do that, we will need some preliminary technical results.
If
\begin{align*}
&\lambda_{m,n}=(m,n,m^3+mn^2-m)
\intertext{and}
&\lambda_{m+k,n+\ell}=\bj{{m+k},n+\ell,(m+k)^3+(m+k)(n+\ell)^2-(m+k)},
\end{align*}
are two distinct elements of
\begin{equation*}
\Lambda=\set{\lambda_{m,n}:=(m,n,m^3+mn^2-m)\in\ZZ^3\ :\ (m,n)\in\ZZ^*\times\ZZ},
\end{equation*}
then $(k,\ell)\in\ZZ^2\setminus\set{(0,0)}$, and
\begin{equation}\label{53}
\lambda_{m+k,n+\ell}-\lambda_{m,n}
=(k,\ell,Q(k,\ell,m,n))
\end{equation}
with
\begin{equation*}
Q(k,\ell,m,n):=k(3m^2+n^2)+2\ell mn
+(3k^2+\ell^2)m+{k(2\ell n+k^2+\ell^2-1).}
\end{equation*}

\begin{lemma}
The polynomial $Q(k,\ell,m,n)$ satisfies the following relation:
\begin{equation}\label{54}
Q(k,\ell,m,n)=k(3m^2+n^2)+2\ell mn+o\bj{m^2+n^2}\qtq{as}\norm{(m,n)}\to+\infty.
\end{equation}
Furthermore, we have
\begin{equation}\label{55}
Q(k,\ell,m,n)=m\ell(2n+\ell)\qtq{if}k=0,
\end{equation}
and
\begin{equation}\label{56}
\begin{split}
3kQ(k,\ell,m,n)=&\bj{3k\bj{m+\frac{k}{2}}+\ell\bj{n+\frac{\ell}{2}}
}^2
\\&-(\ell^2-3k^2)\bj{n+\frac{\ell}{2}}^2+\frac{3k^2}{4}\bj{k^2+\ell^2-4}
\end{split}
\end{equation}
in the general case.
\end{lemma}

\begin{proof}
The relations \eqref{54} and \eqref{55} can be obtained according to the classical procedure to obtain the canonical form a quadratic form. Note that we may write $3kQ(k,\ell,m,n)$ in the form
\begin{equation*}
(3k+\ell n+c_1)^2+c_2(n+c_3)^2+c_4,
\end{equation*}
with suitable constants $c_1,\ldots, c_4$. An elementary computation allows us to determine these constants, leading to the identity \eqref{56}.
\end{proof}

In order to estimate the norms $\lambda_{m+k,n+\ell}-\lambda_{m,n}$ by \eqref{53}, first we investigate the sets
\begin{equation*}
B_{k,\ell,R}:=\set{(m,n)\in\ZZ^*\times\ZZ\ :\ \abs{Q(k,\ell,m,n)}<R},
\end{equation*}
for $(k,\ell)\in\ZZ^2\setminus\set{(0,0)}$ and $R>0$.

\begin{lemma}\label{l55}\mbox{}
Let $(k,\ell)\in\ZZ^2\setminus\set{(0,0)}$ and $R>0$.
\begin{enumerate}[\upshape (i)]
\item If $3k^2>\ell^2$, then the set $B_{k,\ell,R}$ is finite.

\item If $k=0$ and $\ell$ is even, then $B_{k,\ell,R}$ is covered by  a finite set and the  sequence $\bj{m,-\frac{\ell}{2}}$, with $m\in\ZZ^*$.

\item If $k=0$ and $\ell$ is odd, then $B_{k,\ell,R}$ is finite.

\item If $0<3k^2<\ell^2$, then $B_{k,\ell, R}$ is covered by a finite set and  two sequences $(m_1(n),n)_{n\in\ZZ}$ and $(m_2(n),n)_{n\in\ZZ}$ satisfying the relations
\begin{equation}\label{57}
m_j(n)=\alpha_jn+\beta_j+o(1),\quad |n|\to+\infty,\quad j=1,2
\end{equation}
with suitable constants $\alpha_j\ne 0$ and $\beta_j$ (see Figure \ref{fig1} for an example of this case).
\end{enumerate}
\end{lemma}
{It is clear that the four cases of Lemma \ref{l55} correspond to all the possible situations.}

\begin{figure}[H]
\centering
\includegraphics[scale=0.20]{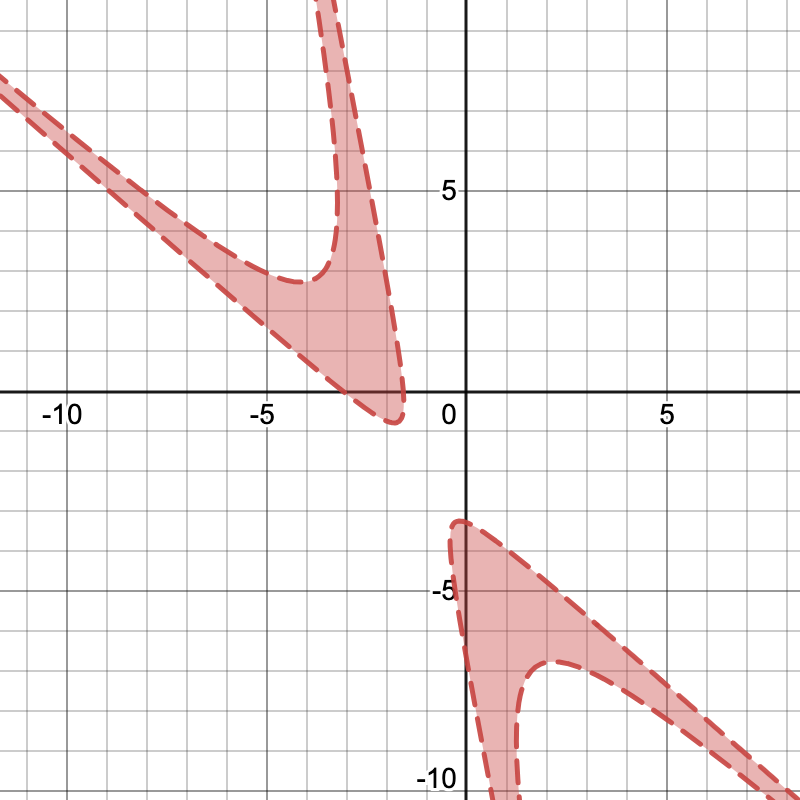}
\caption{(iv) The region containing $B_{2,4,7}$}
\label{fig1}
\end{figure}

\begin{proof}
(i) If $3k^2>\ell^2$, then the quadratic form $k(3m^2+n^2)+2\ell mn$ of $(m,n)$ is (positive or negative) definite.
Then we infer from \eqref{54} the relation
\begin{equation*}
\abs{Q(k,\ell,m,n)}\to+\infty\qtq{as}\norm{(m,n)}\to+\infty,
\end{equation*}
so that $B_{k,\ell,R}$ is a finite set.
\medskip

(ii) If $k=0$, then $\ell\ne 0$ and
\begin{equation*}
Q(k,\ell,m,n)=m\ell (2n+\ell)
\end{equation*}
by \eqref{55}.
This expression vanishes if $\ell$ is even and $n=-\ell/2$, so that $\ZZ^*\times\set{-\ell/2}\subset B_{k,\ell,R}$ in this case.

Since $m\ne 0$ by our assumption, if $(m,n)$ does not belong to $\ZZ^*\times\set{-\ell/2}$, then
\begin{equation*}
\abs{Q(k,\ell,m,n)}\ge \min\set{|m|,|\ell(2n+\ell)|},
\end{equation*}
and the last expression becomes larger than $R$ when $\norm{(m,n)}\to+\infty$ is sufficiently large.
Therefore $B_{k,\ell,R}$ is covered by $\ZZ^*\times\set{-\ell/2}$  and a finite set.
\medskip

(iii) We may repeat the proof of (ii).
Since now $\ell$ is odd, $\ZZ^*\times\set{-\ell/2}$ is disjoint from $B_{k,\ell,R}\subset\ZZ^2$, and hence $\ZZ^*\times\set{-\ell/2}$ is not necessary for the covering of $B_{k,\ell,R}$.
\medskip

(iv) Setting
\begin{equation*}
y=3k\bj{m+\frac{k}{2}}+\ell\bj{n+\frac{\ell}{2}}
\qtq{and}
x=\sqrt{\ell^2-3k^2}\bj{n+\frac{\ell}{2}}
\end{equation*}
for brevity, \eqref{56} may be written in the form
\begin{equation*}
3kQ(k,\ell,m,n)
=y^2-x^2+\frac{3k^2}{4}\bj{k^2+\ell^2-4}.
\end{equation*}
Using this equality, the condition $\abs{Q(k,\ell,m,n)}<R$ in the definition of $B_{k,\ell,R}$ implies the  relation
\begin{equation*}
\abs{y^2-x^2}<3|k|R+\frac{3k^2}{4}\abs{k^2+\ell^2-4}=:C,
\end{equation*}
or equivalently
\begin{equation*}
-C<y^2-x^2<C.
\end{equation*}
Observe that $|x|\to+\infty\Longleftrightarrow|y|\to+\infty$ in this region (see Figure \ref{fig2}).
\begin{figure}[H]
\centering
\includegraphics[scale=0.20]{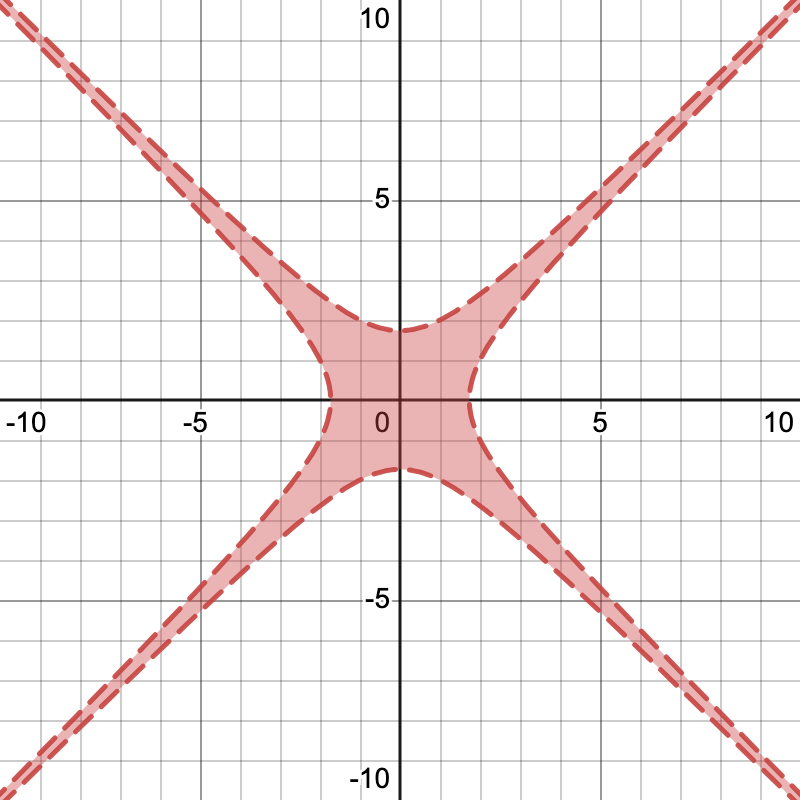}
\caption{The region $-9<y^2-x^2<9$}
\label{fig2}
\end{figure}

Hence either
\begin{align*}
&|x|<\sqrt{C}\qtq{and}|y|<\sqrt{2C},
\intertext{or}
&|x|\ge\sqrt{C}\qtq{and}\sqrt{x^2-C}<|y|<\sqrt{x^2+C}.
\end{align*}
Observe that the  set of these vectors $(x,y)$ shrinks to asymptotes $y=\pm x$ of the corresponding hyperbola because, as $|x|\to +\infty$,
\begin{equation*}
\Bigl\lvert|y|-|x|\Bigr\rvert
<\sqrt{x^2+C}-\sqrt{x^2-C}
=\frac{2C}{\sqrt{x^2+C}+\sqrt{x^2-C}}
\to 0.
\end{equation*}
We may also write it in the following form:
\begin{equation*}
\min\set{\abs{y-x},\abs{y+x}}=o(1)\qtq{as}|x|\to+\infty.
\end{equation*}
Returning to the variables $m, n$, dividing by $3k$, and introducing the sequences
\begin{equation*}
\widetilde{m_1}(n):=\frac{\sqrt{\ell^2-3k^2}-\ell}{3k}\bj{n+\frac{\ell}{2}}-\frac{k}{2}
\qtq{and}
\widetilde{m_2}(n):=\frac{-\sqrt{\ell^2-3k^2}-\ell}{3k}\bj{n+\frac{\ell}{2}}-\frac{k}{2},
\end{equation*}
the last relation becomes
\begin{equation}\label{58}
\min\set{\abs{m-\widetilde{m_1}(n)},\abs{m-\widetilde{m_2}(n)}}=o(1)\qtq{as}|n|\to+\infty;
\end{equation}
because the relation $|n|\to+\infty$ implies $|x|\to+\infty$.
If $|n|$ is sufficiently large, say $|n|>N$, then
\begin{equation*}
\min\set{\abs{m-\widetilde{m_1}(n)},\abs{m-\widetilde{m_2}(n)}}<\frac{1}{2}.
\end{equation*}
Observe that at most two integers $m$ may satisfy this inequality: the closest integers to $\widetilde{m_1}(n)$ and $\widetilde{m_2}(n)$.
Denoting these closest integers by $m_1(n)$ and $m_2(n)$ if $|n|>N$, and defining $m_j(n):=\widetilde{m_j}(n)$ otherwise, the set
\begin{equation*}
\set{(m,n)\in B_{k,\ell,R}\ :\ |n|>N}
\end{equation*}
is covered by the two sequences $(m_1(n),n)_{n\in\ZZ}$ and $(m_2(n),n)_{n\in\ZZ}$.
Furthermore,  in view of \eqref{58} and the definition of the sequences $(\widetilde{m_j}(n))$ the relations \eqref{57} are satisfied with
\begin{equation*}
\alpha_j=\frac{\pm\sqrt{\ell^2-3k^2}-\ell}{3k}
\qtq{and}
\beta_j=\frac{\alpha_j\ell-k}{2},\quad j=1,2.\end{equation*}
Note that $\alpha_j\ne 0$.

To complete the proof it remains to show that  the set
\begin{equation*}
\set{(m,n)\in B_{k,\ell,R}\ :\ |n|\le N}
\end{equation*}
is finite.
This follows by observing that for any fixed $n$,  by \eqref{56} at most finitely many integers $m$ may satisfy the inequality $\abs{Q(k,\ell,m,n)}<R$, showing the item (iv).
\end{proof}

As a consequence of the previous lemma, we have the following:

\begin{corollary}\label{c56}
The sets
\begin{equation*}
\Lambda_{k,\ell,R}:=\set{\lambda_{m,n}\ :\ (m,n)\in B_{k,\ell,R}}
\end{equation*}
are sparse for all $(k,\ell)\in\ZZ^2\setminus\set{(0,0)}$ and $R>0$.
\end{corollary}

\begin{proof}
We distinguish the four cases of Lemma \ref{l55}.
\medskip

(i) If $3k^2>\ell^2$, then the set $B_{k,\ell,R}$ is finite, so that $\Lambda_{k,\ell,R}$ is finite, too, and hence it is sparse.
\medskip

(ii) If $k=0$, then $B_{k,\ell,R}$ is covered by  a finite set and the sequence
\begin{equation*}
\bj{m,-\frac{\ell}{2}},\quad m\in\ZZ^*.
\end{equation*}
Hence it is sufficient to prove that the set
\begin{equation*}
\set{\bj{m,-\frac{\ell}{2},m^3+\frac{m\ell^2}{4}-1}\ :\ m\in\ZZ}
\end{equation*}
is sparse.
This follows from Lemma \ref{l53} with $N=3$ and $f(m)=m^3+\frac{m\ell^2}{4}-1$.
The assumption \eqref{52} of the lemma is satisfied because
\begin{equation*}
f(m+1)-f(m)=3m(m+1)+1+\frac{\ell^2}{4}\to +\infty\qtq{as}|m|\to +\infty.
\end{equation*}
\medskip

(iii) As in case (i), $B_{k,\ell,R}$ is finite, so that $\Lambda_{k,\ell,R}$ is finite, and hence sparse.
\medskip

(iv) Similarly to the case (ii), it suffices to show that the two sets
\begin{equation*}
S_j:=\set{(m_j(n),n,m_j(n)^3+m_j(n)n^2-m_j(n))\ :\ n\in\ZZ},\quad j=1,2
\end{equation*}
are sparse.
We are going to apply Lemma \ref{l53}, so let us rewrite these sets in the form
\begin{equation*}
S_j=\set{(g_j(n),f_j(n))\ :\ n\in\ZZ}
\end{equation*}
with
\begin{equation*}
g_j(n)=(m_j(n),n)\qtq{and}f_j(n)=m_j(n)^3+m_j(n)n^2-m_j(n).
\end{equation*}
Since the following treatment is identical for both values of $j$, we omit the subscript $j$ for simplicity.
By Lemma \ref{l55} (iv) there exist two constants $\alpha\ne 0$ and $\beta$ such that
\begin{equation*}
m(n)=\alpha n+\beta+o(1)\qtq{as}|n|\to+\infty.
\end{equation*}
Using this we obtain the relation (henceforth all relations are understood for $|n|\to+\infty$)
\begin{equation*}
f(n)=m(n)^3+m(n)n^2-m(n)
=(\alpha^3+\alpha)n^3+(3\alpha^2\beta+\beta)n^2+o(n^2).
\end{equation*}
It follows that
\begin{equation*}
f(n+1)=(\alpha^3+\alpha)(n^3+3n^2)+(3\alpha^2\beta+\beta)n^2+o(n^2),
\end{equation*}
and therefore
\begin{equation*}
f(n+1)-f(n)=3(\alpha^3+\alpha)n^2+o(n^2).
\end{equation*}
Since $3(\alpha^3+\alpha)\ne 0$ by our condition $\alpha\ne 0$, the assumption \eqref{52} of Lemma \ref{l53} is satisfied, and we conclude that $S_1$ and $S_2$ are sparse sets.
\end{proof}

Now we are ready to prove the main proposition of this section which ensures that Theorem \ref{t14} holds.

\begin{proof}[Proof of Proposition \ref{p51}]
Fix a real number $R>0$, and consider the sets
$\Lambda_{k,\ell,R}$, where $(k,\ell)\in\ZZ^2\setminus\set{(0,0)}$, $|k|<R$ and $|\ell|<R$.
They are sparse by Corollary \ref{c56}, therefore their union $\Lambda'$ is also sparse.
By Lemma \ref{l52} it remains to show that $\gamma(\Lambda\setminus \Lambda')\ge R$.

If $\lambda_{m+k,n+\ell}$ and $\lambda_{m,n}$ are two distinct elements of $\Lambda\setminus \Lambda'$, then (see \eqref{53})
\begin{equation*}
\norm{\lambda_{m+k,n+\ell}-\lambda_{m,n}}
=\norm{(k,\ell,Q(k,\ell,m,n))}
\ge\max\set{|k|, |\ell|, |Q(k,\ell,m,n)|}.
\end{equation*}
This obviously yields $\norm{\lambda_{m+k,n+\ell}-\lambda_{m,n}}\ge R$ if $|k|\ge R$ or $|\ell|\ge R$.
Otherwise, we infer from the definition of $\Lambda'$ that
\begin{equation*}
\norm{\lambda_{m+k,n+\ell}-\lambda_{m,n}}
\ge |Q(k,\ell,m,n)|\ge R.
\end{equation*}
This proves the relation $\gamma(\Lambda\setminus \Lambda')\ge R$.
\end{proof}

\section{A review on observation, control and stabilization}\label{s6}

Consider a linear evolutionary problem
\begin{equation}\label{61}
z'=\mathcal{A}z+\mathcal{B}v,\quad z(0)=z_0,
\end{equation}
where $\mathcal{A}$ is a densely defined, closed linear operator in some Hilbert space $H$, and $\mathcal{B}$ is a bounded linear operator of another Hilbert space $G$ into $D(\mathcal{A}^*)'$.
Let us also consider the dual problem
\begin{equation}\label{62}
u'=-\mathcal{A}^*u,\quad u(0)=u_0,\quad \psi=\mathcal{B}^*u.
\end{equation}
Here $\mathcal{A}^*, \mathcal{B}^*$ denote the adjoints of $\mathcal{A}$ and $\mathcal{B}$.

In control-theoretical terminology, $\mathcal{B}$ is a control operator, $u$ is a control, and $\mathcal{B}^*$ is an observability operator.
Assume that the following three hypotheses are satisfied (we denote by $G', H'$ the (anti)dual spaces of $G$ and $H$, respectively):
\begin{itemize}
\item[(H1)] The operator $\mathcal{A}^*$ generates a group $e^{s\mathcal{A}^*}$ in $H'$;

\item[(H2)] We have $D(\mathcal{A}^*)=D(\mathcal{B}^*)$, and there exist two numbers $\lambda\in\CC$ and $c\in\RR$ such that
\begin{equation*}
\norm{\mathcal{B}^*u}\le c\norm{(\mathcal{A}+\lambda I)^*u}\qtq{for all}u_0\in D(\mathcal{A}^*).
\end{equation*}

\item[(H3)] There exist two  numbers $T', c'>0$ such that
\begin{equation*}
\norm{\mathcal{B}^*u}_{L^2(0,T';G')}
\le c'\norm{u_0}_{H'}\qtq{for all}\varphi_0\in D(\mathcal{A}^*).
\end{equation*}
\end{itemize}

It follows from (H1) that for every $u_0\in H'$ the equation in \eqref{62} has a unique solution $u\in C(\RR, H')$, and that
$u\in C(\RR, D(\mathcal{A}^*))\cap C^1(\RR, H')$ if $u_0\in D(\mathcal{A}^*)$. Moreover, note that hypothesis (H2) is a weakening of the boundedness of the operator $\mathcal{B}$; it is satisfied in many problems where the observation is made on the boundary or some non-open set of the underlying domain like in the case of Theorem \ref{t11}.

It follows from hypotheses (H1) and (H2) that hypothesis (H3) is meaningful. It is a (hidden) regularity property, allowing us to extend the observation to $\psi\in L^2_{\text{loc}}(\RR, G')$ for every $u_0\in H'$. Moreover, (H1)--(H3) imply the following strengthened version of (H3): for every $T>0$ there exists a constant $c_T$ such that the solutions of \eqref{62} satisfy the estimates
\begin{equation}\label{63}
\norm{\psi}_{L^2(-T,T;G')}
\le c_T\norm{u_0}_{H'}\qtq{for all}u_0\in H'.
\end{equation}
This result also shows that the value of $T'$ in (H3) has no importance: if it is satisfied for some $T'>0$, then, in fact, it is satisfied for \emph{every} $T'>0$.

Next, we define the solutions of \eqref{61} by transposition. Fix $z_0\in H$ and $v\in L^2_{\text{loc}}(\RR, G)$  arbitrarily.
Multiply the equation in \eqref{61} by the solution $u$ of the equation in \eqref{62}. Integrating by parts formally between $0$ and $T\in\RR$, we easily obtain the identity
\begin{equation}\label{64}
\vv{z(T),u(T)}_{H,H'}
=\vv{z_0,u_0}_{H,H'}
+\int_0^T\vv{z(S),\psi(S)}_{G,G'}\, ds.
\end{equation}

\begin{definition}
We define a solution of \eqref{61} as a  function $z:\RR\to H$ satisfying the identity \eqref{64} for all $u_0\in H'$ and $T\in\RR$.
\end{definition}

The previous definition is justified by the following proposition:

\begin{proposition}
Given $z_0\in H$ and $v\in L^2_{\text{loc}}(\RR, G)$  arbitrarily, the problem \eqref{61} has a unique solution.
Moreover, $z:\RR\to H$ is continuous, and satisfies for every $T>0$ the following estimate:
$$
\norm{z}_{L^{\infty}(-T,T; H)}
\le c_T\bj{\norm{z_0}_H+\norm{v}_{L^2(-T,T; G)}}.
$$
\end{proposition}

\begin{definitions}Now we introduce two definitions:
\begin{itemize}
\item The system \eqref{61} is \emph{exactly controllable} in time $T>0$ if for  every pair of data $z_0, z_1\in H$ there exists a function $v\in L^2_{\text{loc}}(\RR, G)$ such that
\begin{equation*}
\norm{v}_{L^2(0,T; G)}\le c'\bj{\norm{z_0}_H+\norm{z_1}_H},
\end{equation*}
and that the solution of \eqref{61} satisfies the final condition $z(T)=z_1$.

\item The system \eqref{62} is \emph{observable} in time $T>0$ if its solutions satisfy the estimates
$$
\norm{u_0}_{H'}\le c''\norm{\psi}_{L^2(0,T; G')}\qtq{for all}u_0\in H'.
$$
\end{itemize}
\end{definitions}
We note that, similarly to \eqref{63}, by  a density argument the observability is equivalent to the following, apparently weaker condition:
\begin{equation}\tag{H4}
\norm{u_0}_{H'}\le c''\norm{\psi}_{L^2(0,T; G')}\qtq{for all}u_0\in D(A^*).
\end{equation}

The following important duality result extends a classical finite-dimensional theorem; it was essentially proved in \cite{DolRus-1977}:

\begin{theorem}\label{t62}
Assume (H1)--(H3).
For any given $T>0$, the system \eqref{61} is exactly controllable if and only if the system \eqref{62} is observable.
\end{theorem}

Next, we state from \cite{Kom-1997} a rapid stabilization theorem that extended a former result of Slemrod \cite{Slemrod-1974} on bounded control operators:

\begin{theorem}\label{t63}
Assume (H1)--(H4) for some $T>0$, and fix $\omega>0$ arbitrarily.
There exists a linear map $F_{\omega}:H\to G$ and a constant $C>0$ such that
the  closed-loop problem
\begin{equation*}
z'=\mathcal{A}z+\mathcal{B}F_{\omega}z, \quad z(0)=z_0,
\end{equation*}
is well-posed in $H$, and its solutions satisfy the estimate
\begin{equation*}
\norm{z(t)}_H \leq C\norm{z_0}_H e^{-\omega t},
\end{equation*}
for all $z_0 \in H$ and $t \geq 0$.
\end{theorem}

\section{Controllability and stabilization results}\label{s7}

With the presentation of the previous review, we are in a position to prove Theorems \ref{t15}--\ref{t18} by applying Theorems \ref{t62} and \ref{t63}.

\begin{proof}[Proof of Theorems \ref{t17} and \ref{t18}] These theorems will follow with the following choices:
\begin{align*}
&H':=L^2_x(\TT\times\TT),\quad
G':=L^2(\Omega);\\
&\mathcal{A}^*u:=u_x+u_{xxx}+ u_{xyy},\quad
\mathcal{B}^*u:=u|_{\Omega};
\end{align*}
and
\begin{align*}
&D(\mathcal{A}^*)=D(\mathcal{B}^*)=\set{u\in H'\ :\ u_x+u_{xxx}+ u_{xyy}\in H'}.
\end{align*}
Using the Fourier expansion
\begin{equation}\label{71}
u(x,y)=\sum_{(m,n)\in\ZZ^*\times\ZZ}c_{m,n}e^{i(mx+ny)}
\end{equation}
of the elements of $H'$ and Parseval's formula
\begin{equation*}
\norm{u}_{H'}^2=4\pi^2\sum_{(m,n)\in\ZZ^*\times\ZZ}\abs{c_{m,n}}^2,
\end{equation*}
we may show that an element $u\in H'$ belongs to $D(\mathcal{A}^*)=D(\mathcal{B}^*)$ if and only if
\begin{equation}\label{72}
\sum_{(m,n)\in\ZZ^*\times\ZZ}\abs{m(m^2+n^2-m)}^2\abs{c_{m,n}}^2<+\infty.
\end{equation}
Indeed, using formula \eqref{71} we have
\begin{equation}\label{73}
(u_x+u_{xxx}+ u_{xyy})(x,y)=\sum_{(m,n)\in\ZZ^*\times\ZZ}m(m^2+n^2-m)c_{m,n}e^{i(mx+ny)},
\end{equation}
whence $u_x+u_{xxx}+ u_{xyy}\in L^2(\TT\times\TT)$ if and only if the condition \eqref{72} is satisfied.
Furthermore, since the index $m=0$ is missing in \eqref{73} by our assumption $u\in H'$,  $u_x+u_{xxx}+ u_{xyy}$ also belongs to $H'$ if and only if it belongs to $L^2(\TT\times\TT)$.

Hypothesis (H1) follows at once from the representation \eqref{71}.
The following computation shows that hypothesis (H2) is satisfied with $\lambda=1$:
\begin{align*}
\norm{(I+\mathcal{A}^*)u}_{H'}^2
&=\int_{\TT\times\TT}\abs{u+u_x+u_{xxx}+u_{xyy}}^2\, dx\, dy\\
&=\int_{\TT\times\TT}\abs{\sum_{(m,n)\in\ZZ^*\times\ZZ}\bj{1+im(m^2+n^2-1)}c_{m,n}e^{i(mx+ny)}}^2\, dx\, dy\\
&=4\pi^2\sum_{(m,n)\in\ZZ^*\times\ZZ}
\abs{1+im(m^2+n^2-1)}^2\abs{c_{m,n}}^2\\
&\ge 4\pi^2\sum_{(m,n)\in\ZZ^*\times\ZZ}
\abs{c_{m,n}}^2
=\norm{u}_{H'}^2
=\int_{\TT\times\TT}\abs{u}^2\, dx\, dy\\
&\ge \int_{\Omega}\abs{u}^2\, dx\, dy
=\norm{\mathcal{B}^*u}_{G'}^2.
\end{align*}
Hypotheses (H3) and (H4) are equivalent to the two-sided estimates of Theorem \ref{t15} with $U:=\Omega\times (0,T)$.
We may thus applying Theorems \ref{t62} and \ref{t63} to obtain Theorems \ref{t17} and \ref{t18}, respectively.
\end{proof}

\begin{proof}[Proof of Theorems \ref{t15} and \ref{t16}] Now, we make the following choices:
\begin{align*}
&H':=L^2(\TT\times\TT),\quad
G':=L^2(\set{x_0}\times\TT);\\
&\mathcal{A}^*u:=u_x+u_{xxx}+ u_{xyy},\quad
\mathcal{B}^*u:=u|_{\set{x_0}\times\TT};
\end{align*}
and
\begin{align*}
&D(\mathcal{A}^*)=D(\mathcal{B}^*)=\set{u\in H'\ :\ u_x+u_{xxx}+ u_{xyy}\in H'}.
\end{align*}
Note that $L^2(\set{x_0}\times\TT)$ is isometrically isomorphic to (and hence may be identified with) $L^2(\TT)$.

Hypothesis (H1) now follows from the representation
$$
u(x,y)=\sum_{(m,n)\in\ZZ^2}c_{m,n}e^{i(mx+ny)}
$$
of the elements of $H'$ and Parseval's formula
\begin{equation*}
\norm{u}_{H'}^2=4\pi^2\sum_{(m,n)\in\ZZ^2}\abs{c_{m,n}}^2.
\end{equation*}

For the verification of hypothesis (H2) first, we observe that
\begin{equation}\label{75}
\norm{(I+\mathcal{A}^*)u}_{H'}^2
=4\pi^2\sum_{(m,n)\in\ZZ^2}
\abs{1+im(m^2+n^2-1)}^2\abs{c_{m,n}}^2
\end{equation}
because this part of the preceding similar computation remains valid without the assumption $m\ne 0$.

Next, we apply the Cauchy--Schwarz inequality to obtain the following relations, where we write
\begin{equation*}
\mu_{m,n}:=\abs{1+im(m^2+n^2-1)}
\end{equation*}
for brevity:
\begin{align*}
\norm{\mathcal{B}^*u}_{G'}^2
&=\int_{\TT}\abs{\sum_{(m,n)\in\ZZ^2}c_{m,n}e^{i(mx_0+ny)}}^2\, dy
=2\pi\sum_{n\in\ZZ}\abs{\sum_{m\in\ZZ}c_{m,n}e^{imx_0}}^2\\
&\le 2\pi\sum_{n\in\ZZ}\bj{\bj{\sum_{m\in\ZZ}\abs{c_{m,n}e^{imx_0}\mu_{m,n}}^2}
\cdot\bj{\sum_{m\in\ZZ}\frac{1}{\mu_{m,n}^2}}}\\
&=2\pi\sum_{n\in\ZZ}\bj{\bj{\sum_{m\in\ZZ}\abs{c_{m,n}\mu_{m,n}}^2}
\cdot\bj{\sum_{m\in\ZZ}\frac{1}{\mu_{m,n}^2}}}.
\end{align*}
Since $\mu_{0,n}=1$, and $\mu_{m,n}\ge \abs{m}$ if $m\in\ZZ^*$, we have
\begin{equation*}
\sum_{m\in\ZZ}\frac{1}{\mu_{m,n}^2}
\le 1+2\sum_{k=1}^{\infty}\frac{1}{k^2}
<5,
\end{equation*}
for every $n\in\ZZ$, and therefore we infer from the preceding estimate that
\begin{equation*}
\norm{\mathcal{B}^*u}_{G'}^2
\le 10\pi\sum_{(m,n)\in\ZZ^2}\abs{1+im(m^2+n^2-1)}^2\abs{c_{m,n}}^2.
\end{equation*}
Combining with \eqref{75} and noticing that $10\pi<4\pi^2$,  we obtain that
\begin{equation*}
\norm{\mathcal{B}^*u}_{G'}\le \norm{(I+\mathcal{A}^*)u}_{H'},
\end{equation*}
i.e., hypothesis (H2) is satisfied.

Finally, hypotheses (H3) and (H4) are equivalent to the two-sided estimates of Theorem \ref{t14}.
Applying Theorems \ref{t62} and \ref{t63},  Theorems \ref{t17} and \ref{t18} follow.
\end{proof}

\section{Further comments}\label{s8}

In this paper, we considered only internal observability and controllability properties of the system \eqref{11}, however boundary observability and controllability results may also be obtained by adapting our approach; see e.g., \cite{KomMia-2014} and \cite{KomLor-2014} for analogous results on membrane and plate models.

Concerning Theorems \ref{t11}, \ref{t12} and \ref{t14} it is natural to ask whether they remain valid in more general situations, for example, if the segments are oblique or they are moving in time (moving observation). For linear plate models, these questions were studied in \cite{JamKom-2020}, and they probably may be adapted to the Zakharov--Kuznetsov equation \eqref{11}, however, these issues are still open to study.

Note that the method developed in this work was based on the notion of sparse sets, it may be adapted to more general equations, containing further higher-order terms, in the two-dimensional cases.

\subsection*{Acknowledgments}
The authors express their gratitude to the referees for their very thorough evaluation of our manuscript and for many helpful comments and suggestions.
Part of this work was done during the stay of the second author at the Department of Mathematics of the Federal University of Rio de Janeiro (UFRJ) and the Department of Mathematics of the Federal University of Pernambuco (Recife) in September 2023.
He thanks the departments for their hospitality.

\section*{Statements \& Declarations}

\subsection*{Funding} 
The first author was supported by CAPES grant numbers 88881.311964/2018-01 and
 88881.520205/2020-01,  CNPq grant numbers 307808/2021-1 and  401003/2022-1,  MATHAMSUD grant 21-MATH-03 and Propesqi (UFPE). 
 The second author was supported by the following grants:  FACEPE APV-0002-1.01/23, Coordenação de Aperfeiçoamento de Pessoal de Nível Superior (CAPES) -- finance Code 001 and NSFC no.~11871348.  
 The third author was partially supported by CNPq (Brasil) and CAPES (Brasil) -- finance Code 001.


\subsection*{Competing Interests} 
The authors have no relevant financial or non-financial interests to disclose.

\subsection*{Author Contributions}
All authors contributed equally to this work.
All authors read and approved the final manuscript.

\end{document}